\documentclass{amsart}
\usepackage{amscd,amssymb,amsopn,amsmath,amsthm,graphics,amsfonts,accents,enumerate,verbatim,calc}
\usepackage[dvips]{graphicx}
\usepackage[colorlinks=true,linkcolor=red,citecolor=blue]{hyperref}
\usepackage[all]{xy}

\usepackage{tikz}

\calclayout

\newcommand{\rt}{\rightarrow}
\newcommand{\lrt}{\longrightarrow}

\newcommand{\va}{\vartheta}
\newcommand{\st}{\stackrel}

\newcommand{\la}{\lambda}
\newcommand{\La}{\Lambda}
\newcommand{\Ga}{\Gamma}
\newcommand{\ga}{\gamma}

\newcommand{\N}{\mathbb{N} }

\newcommand{\CA}{\mathcal{A} }
\newcommand{\CC}{\mathcal{C} }

\newcommand{\CF}{\mathcal{F} }
\newcommand{\CG}{\mathcal{G} }
\newcommand{\CI}{\mathcal{I} }

\newcommand{\CK}{\mathcal{K} }
\newcommand{\CL}{\mathcal{L} }

\newcommand{\CP}{\mathcal{P} }
\newcommand{\CQ}{\mathcal{Q} }
\newcommand{\CR}{\mathcal{R} }

\newcommand{\CT}{\mathcal{T} }
\newcommand{\CX}{\mathcal{X} }
\newcommand{\CY}{\mathcal{Y} }

\newcommand{\CZ}{\mathcal{Z} }
\newcommand{\CB}{\mathcal{B} }

\newcommand{\PP}{\mathbf{P}}

\newcommand{\mmod}{{\rm{{mod\mbox{-}}}}}
\newcommand{\mmodd}{{\rm{mod}}_0\mbox{-}}

\newcommand{\inj}{{\rm{inj}\mbox{-}}}

\newcommand{\prj}{{\rm{prj}\mbox{-}}}
\newcommand{\GPrj}{\rm{{GPrj}\mbox{-}}}
\newcommand{\Gprj}{\rm{{Gprj}\mbox{-}}}
\newcommand{\GInj}{{\rm{GInj}\mbox{-}}}

\newcommand{\Aus}{{\rm{\bf{Aus}}}}

\newcommand{\im}{{\rm{Im}}}

\newcommand{\add}{{\rm{add}}}

\newcommand{\gen}{{\rm{gen}}}

\newcommand{\cogen}{{\rm{cogen}}}

\newcommand{\pd}{{\rm{pd}}}
\newcommand{\id}{{\rm{id}}}

\newcommand{\Gdim}{{\rm{Gdim}}}

\newcommand{\domdim}{{\rm{domdim}}}

\newcommand{\Coker}{{\rm{Coker}}}
\newcommand{\Ker}{{\rm{Ker}}}

\newcommand{\Hom}{{\rm{Hom}}}
\newcommand{\Ext}{{\rm{Ext}}}
\newcommand{\End}{{\rm{End}}}

\newcommand{\TTF}{\textsf{TTF}}

\theoremstyle{plain}
\newtheorem{theorem}{Theorem}[section]
\newtheorem{corollary}[theorem]{Corollary}
\newtheorem{lemma}[theorem]{Lemma}

\theoremstyle{definition}
\newtheorem{definition}[theorem]{Definition}
\newtheorem{example}[theorem]{Example}

\newtheorem{remark}[theorem]{Remark}

\newtheorem{s}[theorem]{}

\theoremstyle{plain}

\theoremstyle{definition}

\numberwithin{equation}{section}

\begin{document}

\title[Relative Auslander algebras]{On relative Auslander algebras}

\author[Javad Asadollahi and Rasool Hafezi ]{Javad Asadollahi and  Rasool Hafezi }

\address{Department of Mathematics, University of Isfahan, P.O.Box: 81746-73441, Isfahan, Iran}
\email{asadollahi@ipm.ir, asaddollahi@sci.ui.ac.ir}

\address{School of Mathematics, Institute for Research in Fundamental Sciences (IPM), P.O.Box: 19395-5746, Tehran, Iran }
\email{hafezi@ipm.ir}



\begin{abstract}
Relative Auslander algebras were introduced and studied by Beligiannis. In this paper, we apply intermediate extension functors associated to certain recollements of functor categories to study them. In particular, we study the existence of tilting-cotilting modules over such algebras. As a consequence, it will be shown that two Gorenstein algebras of G-dimension 1 being of finite Cohen-Macaulay type are Morita equivalent if and only if their Cohen-Macaulay Auslander algebras are Morita equivalent.
\end{abstract}

\maketitle


\section{Introduction}
Let $\La$ be an artin algebra and $\mmod\La$ denote the category of finitely generated right $\La$-modules. A subclass $\CX$ of $\mmod\La$ is called of finite type if the set of all iso-classes of indecomposable modules of $\CX$ is finite. If $\mmod\La$ itself is of finite type, then $\La$ is called of finite representation type. Let $\La$ be such algebra and $X$ be an additive generator of $\mmod\La$, i.e. $X$ is the direct sum of representatives of the iso-classes of indecomposable $\La$-modules. Auslander \cite{Aus74} studied the algebra $\End_{\La}(X)$, which is known as the Auslander Algebra of $\La$. Beligiannis \cite[\S 6]{Be}, introduced a relative version of Auslander algebras, with respect to a contravariantly finite resolving subcategory $\CX$ of $\mmod\La$ of finite type. The natural candidate for $\Aus(\CX)$, the $\CX$-relative Auslander algebra of $\La$, is the endomorphism ring of an additive generator of $\CX$.

On the other hand, recently some authors study algebras admitting tilting-cotilting modules. For instance, Crawley-Boevey and Sauter \cite[Lemma 1.1]{CS} showed that an artin algebra $\Ga$ is an Auslander algebra if and only if its global dimension is at most $2$ and the full subcategory $\CC_{\Ga}$ of $\mmod\Ga$ consisting of all modules that are generated and are cogenerated by the direct sum of the representatives of the isomorphism classes of all injective-projective $\Ga$-modules contains a tilting module. Moreover, they showed that in this case, there is a unique basic such module which is also a cotilting module. Later on, without any assumption on the global dimension of $\Ga$, the authors of \cite{NRTZ} showed that dominant dimension of $\Ga$ is at least $2$ if and only if $\CC_{\Ga}$ contains either a tilting or a cotilting module, where by definition, dominant dimension of $\Ga$ is at least $n$ if the first $n$ terms of the minimal injective resolution of the $\Ga$-module $\Ga$ are projectives. They therefore, provided a generalization of \cite[Lemma 1.1]{CS}. Moreover, they
\cite[Theorem 2.4.11]{NRTZ} showed that an artin algebra $\Ga$ is a $1$-Auslander-Gorenstein algebra if and only if $\CC_{\Ga}$ contains a tilting-cotilting module. Recall that the notion of $n$-Auslander-Gorenstein algebras defined recently by Iyama and Solberg \cite{IS} as artin algebras $\Ga$ with the property that $\id\Ga_{\Ga} \leq n+1 \leq \domdim \Ga.$ The most recent attempt in this direction is the work of Pressland and Sauter that, among other results, gave a characterization of $n$-Auslander-Gorenstein algebras as well as $n$-Auslander algebras via shifted and coshifted modules \cite[Theorem 3.9]{PS}.

In this paper, we study the relative Auslander algebras, in the sense of Beligiannis, and use functor categories methods to prove the existence of tilting-cotilting modules in certain cases. To this end, we use the notion of intermediate extension functor associated to a recollement. Let $\CX$ be a contravariantly finite subcategory of $\mmod\La$ containing projective $\La$-modules. In \cite{AHK}, a relative version of Auslander's Formula with respect to $\CX$ has been studied. More precisely, it is shown that for such $\CX$, there exists a recollement
\[\xymatrix{\mmodd \CX \ar[rr]^{\ell}  && \mmod \CX \ar[rr]^{\va} \ar@/^1pc/[ll]^{\ell_{\rho}} \ar@/_1pc/[ll]_{\ell_{\la}} && \mmod \La \ar@/^1pc/[ll]^{\va_{\rho}} \ar@/_1pc/[ll]_{\va_{\la}} }\]
of abelian categories, where $\mmod\CX$ is the abelian category consisting of all finitely presented contravariant functors from $\CX$ to $\CA b$ and $\mmodd\CX$ consists of all functors that vanish on $\La$. We use the intermediate extension functor associated to this recollement, denoted by $\zeta^{\CX}$.
If we assume further that $\CX$ is of finite type, then $\mmod\CX$ is equivalent to $\mmod\Aus(\CX)$ and hence we get a recollement
\[\xymatrix{\mmodd \Aus(\CX) \ar[rr]^{L}  && \mmod \Aus(\CX) \ar[rr]^{V} \ar@/^1pc/[ll]^{L_{\rho}} \ar@/_1pc/[ll]_{L_{\la}} && \mmod \La \ar@/^1pc/[ll]^{V_{\rho}} \ar@/_1pc/[ll]_{V_{\la}} }\]
which is equivalent to the above one. We use this equivalence to transfer the problems from $\Aus(\CX)$ to $\mmod\CX$, use functor categories methods there and then back to our own home. Specializing our results to the class of Gorenstein projective modules over Gorenstein algebras of finite Cohen-Macaulay type have interesting applications. For instance, we show that the condition of being gentle in Theorem 3.10 of \cite{CL} is redundant, see Corollary \ref{CL} below.

Furthermore, we show that the surjective map $\Aus$ defined in \cite{KZ} is injective when we restrict it to the class of Gorenstein algebras of G-dimension $1$. It is easy to see that it can not be injective for Gorenstein algebras of higher G-dimension, see Remark \ref{KZ} below. It is worth noting that Gorenstein algebras of G-dimension 1 are playing an important role in representation theory of finite dimensional algebras, which includes some important classes of algebras, e.g. the cluster-tilted algebras \cite{BMR} and \cite{BMRT}, 2-CY-tilted algebras \cite{KR}, or more generally the endomorphism algebras of cluster tilting objects in triangulated categories \cite{KZh}. See also \cite{GLS}, for more applications of the class of 1-Gorenstein algebras attached to symmetrizable generalized Cartan matrices.

\section{Preliminaries}
In this section, we recall some basic facts we need throughout the paper. Let us begin by recalling the notion of Gorenstein projective modules.

Let $A$ be a ring. A totally acyclic complex of projectives is an acyclic complex $\PP$ of projective $A$-modules such that the induced complex $\Hom_{A}(\PP, Q)$ is acyclic, for every projective $A$-module $Q$. Dually the notion of totally acyclic complex of injectives can be defined. The syzygies of a totally acyclic complex of projectives, respectively of injectives, are called Gorenstein projective, respectively Gorenstein injective, $A$-modules. The class of all Gorenstein projective, resp. Gorenstein injective, modules is denoted by $\GPrj A$, resp. $\GInj A$. We set $\Gprj A= \GPrj A \cap \mmod A$.

\s {\sc Recollements.}
Let $\CA$, $\CA'$ and $\CA''$ be abelian categories. By \cite{BBD} a recollement of $\CA$ with respect to $\CA'$ and $\CA''$, denoted by $\CR(\CA', \CA, \CA'')$ is a diagram
\[\xymatrix{\CA'\ar[rr]^{\ell \ \ }  && \CA \ar[rr]^{\va \ } \ar@/^1pc/[ll]_{\ell_{\rho}} \ar@/_1pc/[ll]_{\ell_{\la}} && \CA'' \ar@/^1pc/[ll]_{\va_{\rho}} \ar@/_1pc/[ll]_{\va_{\la}} }\]
of additive functors satisfying the following conditions:
\begin{itemize}
\item[$(i)$] $(\ell_{\la},\ell)$, $(\ell,\ell_{\rho})$, $(\va_{\la}, \va)$ and $(\va,\va_{\rho})$ are adjoint pairs.
\item[$(ii)$] $\ell$, $\va_{\la}$ and $\va_{\rho}$ are fully faithful.
\item[$(iii)$] $\im \ell= \Ker \va$.
\end{itemize}

\s {\sc Equivalence of recollements.} \label{Eq-Recollement}
Let $\CR(\CA', \CA, \CA'')$ be the above recollement of $\CA$. We say that it is equivalent to the recollement $\CR(\CB', \CB, \CB'')$ below
\[\xymatrix{\CB'\ar[rr]^{\iota \ \ }  && \CB \ar[rr]^{\nu \ } \ar@/^1pc/[ll]_{\iota_{\rho}} \ar@/_1pc/[ll]_{\iota_{\la}} && \CB'' \ar@/^1pc/[ll]_{\nu_{\rho}} \ar@/_1pc/[ll]_{\nu_{\la}} }\]
if there are equivalences $\Phi: \CA \lrt \CB$ and $\Phi'': \CA'' \lrt \CB''$ such that the following diagram
\[\xymatrix{\CA \ar[r]^{\va} \ar[d]_{\Phi} & \CA'' \ar[d]^{\Phi''} \\ \CB \ar[r]^{\nu} & \CB'' }\]
is commutative up to natural equivalences of functor. This is equivalent to say that there exists natural equivalences $\Phi': \CA' \lrt \CB'$ and $\Phi$ and $\Phi''$ above such that all six diagrams associated to the six functors of the recollements commute up to natural equivalences. See Definition 4.1 and Lemma 4.2 of \cite{PV}.

\s {\sc Intermediate extension functor.}\label{InterExtFunc}
Consider the recollement $\CR(\CA', \CA, \CA'')$. Since $(\va_{\la},\va)$ is an adjoint pair, for every $A \in \CA$ and $A'' \in \CA''$ there exists an isomorphism
\[\CA(\va_{\la}(A''),A)=\CA''(A'',\va(A)),\]
of abelian groups. If we set $A:=\va_{\rho}(A'')$, using the fact that the counit of adjunction $\eta:\va\va_{\rho} \lrt \id_{\CA''}$ is an isomorphism, we get a natural transformation
\[\ga: \va_{\la} \lrt \va_{\rho}.\]
The intermediate extension functor $\zeta: \CA'' \lrt \CA$ is defined by
\[\zeta(A'') :=\im(\ga_{A''})=\im(\va_{\la}(A'') \lrt \va_{\rho}(A'')). \]
We refer the reader to \cite[\S 2.1]{CS} for a list of the properties of this functor. In particular, it is shown \cite[Lemma 2.2]{CS} that $\zeta$ is a fully faithful functor that preserves indecomposable objects.

\s {\sc \TTF-triples.}\label{TTF-Triples}
Let $\CA$ be an abelian category. A pair $(\CX,\CY)$ of subcategories of $\CA$ is called a torsion pair if $\Hom_{\CA}(\CX,\CY)=0$ and moreover, for every object $A \in \CA$, there exists a short exact sequence
\[0 \lrt X_A \lrt A \lrt Y^A \lrt 0,\]
such that $X_A \in \CX$ and $Y^A \in \CY$. In this case, $\CX$ is called a torsion class and $\CY$ is called a torsion-free class.
The torsion pair $(\CX,\CY)$ is called hereditary if $\CX$ is closed under subobjects and is called cohereditary if $\CY$ is closed under quotients.

If $(\CX,\CY)$ and $(\CY,\CZ)$ are torsion pairs in $\CA$, then the triple $(\CX,\CY,\CZ)$ is called a torsion torsion-free triple, or simply a \TTF-triple, of $\CA$, see \cite[Definition 2.2]{PV}.

\s {\sc \TTF-triples and recollements.}\label{Rem-Bijection}
By \cite[Corollary 4.4]{PV}, if $\CA$ has enough projective and enough injective objects, then the equivalence classes of recollements of $\CA$, i.e. equivalence classes of recollements with $\CA$ as the middle term, are in bijection with the \TTF-triples in $\CA$. Based on this bijection associated to the recollement $\CR(\CA',\CA,\CA'')$ there exists a \TTF-triple  $(\Ker\ell_{\la},\ell(\CA'),\Ker\ell_{\rho})$ such that $(\Ker\ell_{\la},\ell(\CA'))$ is a cohereditary and $(\ell(\CA'),\Ker\ell_{\rho})$ is a hereditary torsion pair.
On the other hand, a \TTF-triple $(\CT,\CF,\CL)$ in $\CA$ induces a recollement
\[\xymatrix{\CF \ar[rr]^{\ell}  && \CA \ar[rr]^{v} \ar@/^1pc/[ll]^{{\ell}_{\rho}} \ar@/_1pc/[ll]_{\ell_{\la}} && \frac{\CA}{\CF}, \ar@/^1pc/[ll]^{v_{\rho}} \ar@/_1pc/[ll]_{v_{\la}} }\]
where $\ell$ is the inclusion and $v$ is the canonical quotient map. For the details of this bijection we refer the reader to \cite{PV}.

\section{$\CX$-intermediate extension functor}
Let $A$ be a right coherent ring and $\CX$ be a subcategory of $\mmod A$. An additive contravariant functor $F:\CX \rt \CA b$, where $\CA b$ denotes the category of abelian groups, is called a (right) $\CX$-module. An $\CX$-module $F$ is called finitely presented if there exists an exact sequence $$\CX( - ,X) \rt \CX( - ,X') \rt F \rt 0,$$ with $X$ and $X'$ in $\CX$. All finitely presented $\CX$-modules and natural transformations between them from a category that will be denoted by $\mmod\CX$ or sometimes ${\rm fp}(\CX^{op}, \CA b)$. It is known that if $\CX$ is a contravariantly finite subcategory of $\mmod A$ then $\mmod\CX$ is an abelian category, see \cite[\S 2]{AHK}.

\s \label{Our Recollement} Let $A$ be a right coherent ring and $\CX$ be a contravariantly finite subcategory of $\mmod A$ containing $\prj A$. Let $\mmodd\CX$ denote the Serre subcategory of $\mmod\CX$ consisting of all objects that vanish at $\La$ as a right $\La$-module. By \cite[Theorem 3.7]{AHK} we have a recollement
\[\xymatrix{ \mmodd \CX \ar[rr]^{\ell}  && \mmod \CX \ar[rr]^{\va} \ar@/^1pc/[ll]^{\ell_{\rho}} \ar@/_1pc/[ll]_{\ell_{\la}} && \mmod A \ar@/^1pc/[ll]^{\va_{\rho}} \ar@/_1pc/[ll]_{\va_{\la}} }\]

\noindent of abelian categories. As this recollement has a central role in this section, we recall explicitly some functors appearing in this recollement. This recollement will be denoted by $\CR(\CX,A)$.
\begin{itemize}
\item [$\mbox{-}$] $\va$: By \cite[Proposition 3.1]{AHK}, the functor $Y^A: (\mmod\CX, \mmod A) \lrt (\CX, \mmod A)$, that is induced from the Yoneda embedding $Y: \CX \lrt \mmod\CX$,  admits a left adjoint $Y^A_{\la}$. We set $\va:=Y^A_{\la}(\iota)$, where $\iota: \CX \lrt \mmod A$ is the inclusion functor. It is shown that $\va$ is an exact functor and if $\CX( - ,X_1) \st{(-,d)} \lrt \CX( - ,X_0) \lrt F \lrt 0$ is a projective presentation of $F$, then $\va(F)$ is the cokernel of the induced map $X_1 \st{d}{\lrt} X_0$.
\item [$\mbox{-}$] $\va_{\la}$: By \cite[Proposition 3.6]{AHK}, $\va_{\la}$ that is a left adjoint of $\va$, is defined as follows. Let $M \in \mmod \La$ and $P \st{d}{\lrt} Q \st{\varepsilon}{\lrt} M \lrt 0$ be a projective presentation of $M$ in $\mmod A$. Set
    \[\va_{\la}(M):= \Coker(\CX( - ,P) \lrt \CX( - ,Q)).\]

\item [$\mbox{-}$] $\va_{\rho}$: Let $M \in \mmod \La$. Set $\va_{\rho}(M):=\Hom_A( - ,M)\vert_{\CX}=( - ,M)\vert_{\CX}$. By Lemma 3.5 of \cite{AHK}, $\va_{\rho}$ is a full and faithful functor and by Proposition 3.6 of \cite{AHK} $\va_{\rho}$ is a right adjoint of $\va$. In case $M$ belongs to $\CX$, we simply write $\va_{\rho}(M)=\CX(-, M)$ without the $|$ sign.
\end{itemize}

\s {\bf Setup.} \label{Setup} Throughout, we assume that $\CX$ is a contravariantly finite subcategory of $\mmod A$ that contains $\prj A$.

\begin{definition}
As in \ref{InterExtFunc}, we can associate an intermediate extension functor to the recollement $\CR(\CX,A)$. Let us denote it by $\zeta^{\CX}$ and call it the $\CX$-intermediate extension functor. If $\CX=\mmod A$, we denote $\zeta^{\CX}$ by $\zeta^A$.
\end{definition}

\s {\bf Notation.} Let $M \in \mmod A$. For the ease of notation, we denote $\zeta^{\CX}(M) \in \mmod\CX$ by $\zeta^{\CX}_M$. Therefore, for $X \in \CX$, the action of $\zeta^{\CX}_M$ on $X$ will be denoted by $\zeta^{\CX}_M(X)$.

\begin{example}
An important example is $\CX=\Gprj\La$, where $\La$ is an artin algebra with the property that $\Gprj\La$ is a contravariantly finite subcategory of $\mmod\La$. This include, in particular, the class of all virtually Gorenstein algebras \cite[Proposition 4.7]{Be} and also the class of algebras of finite CM-type. Recall that virtually Gorenstein algebras introduced and studied in deep by Beligiannis and Reiten \cite{BR}. For simplicity, the relative intermediate extension functor associated to the recollement $\CR(\Gprj\La, \La)$ will be denoted by $\zeta^{\CG}$.
\end{example}

For an $A$-module $M$, consider a short exact sequence $0 \rt \Omega(M) \rt P \rt M \rt 0$ with $P$ projective. The module $\Omega(M)$ is then called a syzygy module of $M$. Note that syzygy modules of $M$ are not uniquely determined. An $n$th syzygy of $M$ will be denoted by $\Omega^n(M)$, for $n \geq 2$.

\begin{lemma}\label{Lem-Fundamental}
Let $M \in \mmod A$. Then there exits the following exact sequence
\[0 \lrt \CK_M \lrt \va_{\la}(M) \lrt \va_{\rho}(M) \lrt \CL_M \lrt 0,\]
in $\mmod \CX$, where $\CK_M$ and $\CL_M$ belong to $\mmodd\CX$.
\end{lemma}

\begin{proof}
Consider a projective presentation $P \lrt Q \lrt M \lrt 0$ of $X$ in $\mmod A$. In view of the Yoneda Lemma we have the following commutative diagram
\[ \xymatrix@C-0.5pc@R-0.5pc{ &&& 0 \ar[d] & \\ 0 \ar[r] & (-, \Omega^2(M))|_{\CX} \ar[r] \ar@{=}[d] & (-, P) \ar[r] \ar@{=}[d] & (-, \Omega(M))|_{\CX} \ar[r] \ar[d] & \CK_M \ar[r] & 0 \\
0\ar[r] & (-, \Omega^2(M))|_{\CX} \ar[r] & (-, P) \ar[r] & (-, Q) \ar[d] \ar[r] & \va_{\la}(M) \ar[r] \ar@{.>}[dl]^{\gamma_M} & 0 \\
&&& (-, M)|_{\CX} \ar[d] & & \\ &&& \CL_M \ar[d] &&\\
&&& 0 && }\]\label{Diagram}
Applying Snake lemma to the diagram
\[ \xymatrix{  0 \ar[r] & ( - ,P) \ar@{=}[r] \ar[d] & ( - , P) \ar[r] \ar[d] & 0 \ar[r] \ar[d]  & 0 \\\
0 \ar[r] & ( - ,\Omega(M))|_{\CX}  \ar[r] & ( - ,Q) \ar[r] & ( - ,M)|_{\CX} }\]
and using the fact that $\va_{\rho}(M)=( - ,M)|_{\CX}$, we get the exact sequence
\[\xymatrix{0 \ar[r] & \CK_M \ar[r] & \va_{\la}(M) \ar[r]^{\gamma_M} & \va_{\rho}(M) \ar[r] & \CL_M \ar[r] & 0.}\]
It follows from definition that $\CK_M$ and $\CL_M$ both belong to $\mmodd A$. Hence the proof is complete.
\end{proof}

\begin{remark}\label{Rem-Triangle}
By definition of $\zeta^{\CX}(X)$, we get the following two short exact sequences
\[\xymatrix{0 \ar[r] & \CK_M  \ar[r] & \va_{\la}(M) \ar[r] & \zeta^{\CX}_M \ar[r] & 0,}\]
and
\[\xymatrix{0 \ar[r] & \zeta^{\CX}_M \ar[r] & \va_{\rho}(M) \ar[r] & \CL_M \ar[r] & 0.}\]
\end{remark}

\begin{remark}\label{Interpretation}
The sequence of the above lemma also can be obtained using the units and counits of adjunctions of the functors appearing in the recollement $\CR(\CX,A)$, see \cite[Proposition 2.8]{PV}. In fact, for every functor $F$ in $\mmod\CX$, we have the following exact sequence
\[0 \lrt \ell\ell_{\rho}(F) \lrt F \st{\delta^{{\va}_{\rho}\va}_F} \lrt {\va}_{\rho}{\va}(F) \lrt {\rm{Coker}}{\delta^{{\va}_{\rho}{\va}}_F} \lrt 0,\]
where in view of the definition of the functors involved, this sequence has the following simpler form
\[0 \lrt F_0 \lrt F \lrt ( - ,\va(F))\vert_{\CX} \lrt F_1 \lrt 0,\]
where $F_0$ and $F_1$ are in $\mmodd\CX$. In particular, for every $M \in \mmod \La$, if we set $F=\va_{\la}(M)$, we get the following exact sequence
\[0 \lrt \CK_M \lrt \va_{\la}(M) \lrt ( - ,\va(\va_{\la}(M)))\vert_{\CX} \lrt \CL_M \lrt 0.\]
But $( - ,\va(\va_{\la}(M)))\vert_{\CX}=\va_{\rho}(M)$, because $\va(\va_{\la}(M))=M$. So we get the desired sequence.
\end{remark}

\begin{lemma}\label{Lem-pd}
Let $\CX$ be as in Setup \ref{Setup}.
\begin{itemize}
\item [$(i)$] If $\CX$ is closed under syzygies, then for every $X \in \CX$, projective dimension of $\zeta^{\CX}_X$ as an object of $\mmod\CX$ is at most one. In particular, if $X \in \prj A$, then $\zeta^{\CX}_X=\va_{\rho}(X)=( - ,X)$ is projective in $\mmod \CX$.
\item [$(ii)$] If $\CX$ is closed under submodules, then for every $M \in \mmod\La$, projective dimension of $(-, M)|_{\CX}$ as an object of $\mmod \CX$ is at most one.
\end{itemize}
\end{lemma}

\begin{proof}
$(i).$ Let $X \in \CX$ and consider a projective presentation $P \st{d} \rt  Q \rt X \rt 0$ of $X$. The commutative diagram of the proof of Lemma \ref{Lem-Fundamental} for $M=X$ yields the exact sequence
\[0 \lrt ( - ,\Omega(X)) \lrt ( - ,Q) \lrt \zeta^{\CX}_X \lrt 0.\]
Since $\CX$ contains projectives and is closed under syzygies, we deduce that projective dimension of $\zeta^{\CX}_X$ is at most one.

Now assume that $X \in \prj A$. Then the argument above implies that $\zeta^{\CX}_X=( - ,X)$ and so is a projective object of $\mmod \CX$.

$(ii).$ Pick $M \in \mmod\La$. Since $\CX$ is contravariantly finite in $\mmodd\La$, there is a right $\CX$-approximation $f:X \rt M$. This induces the short exact sequence
$$0 \rt (-, \Ker(f)) \rt (-, X) \rt (-, M)|_{\CX} \rt 0$$ in $\mmod \CX.$ Since $\CX$ is closed under submodules, $(-, \Ker(f))$ is a projective object and so $\rm{pd}(-, M)|_{\CX}\leq 1.$
\end{proof}

\begin{lemma}\label{Lem-rigid}
Let $\CX$ be as in Setup \ref{Setup} which is also closed under syzygies and is contained in ${}^{\perp}A$. Then for every $X_1$ and $X_2$ in $\CX$, $$\Ext^1(\zeta^{\CX}_{X_1}, \zeta^{\CX}_{X_2})=0.$$ In particular, for every $X \in \CX$, $\zeta^{\CX}_X$ is rigid.
\end{lemma}

\begin{proof}
Fix $X_1$ and $X_2$ in $\CX$. Consider the projective resolution
\[0 \lrt ( - ,\Omega(X_1)) \lrt ( - ,Q) \lrt \zeta^{\CX}_{X_1} \lrt 0,\]
of $\zeta^{\CX}_{X_1}$ in $\mmod \CX$, where $P \lrt Q \lrt X_1 \lrt 0$ is a projective presentation of $X_1$. To prove the result, we should show that the morphism
\[(( - , Q), \zeta^{\CX}_{X_2}) \lrt (( - ,\Omega(X_1)), \zeta^{\CX}_{X_2})\]
is surjective. Let $P' \lrt Q' \lrt X_2 \lrt 0$ be a projective presentation of $X_2$. Since $\CX$ is closed under syzygies, the morphism $\varphi: ( - ,\Omega(X_1)) \lrt \zeta^{\CX}_{X_2}$ lifts to a morphism $\varphi': ( - ,\Omega(X_1)) \lrt ( - , Q')$. The Yoneda Lemma now induces the following diagram
\[\xymatrix{0 \ar[r] & \Omega(X_1) \ar[r] \ar[d] & Q \\ & Q'}\]
Since $\CX \subseteq {}^{\perp}\La$, the morphism $\Omega(X_1) \lrt Q'$ extends to a morphism $Q \lrt Q'$. This in turn induces a morphism
\[( - ,Q) \lrt ( - ,Q').\]

It is now plain that the composition $( - ,Q) \lrt ( - ,Q') \lrt \zeta^{\CX}_{X_2}$ maps to $\varphi$. This completes the proof.
\end{proof}

\begin{lemma}\label{Lem-id}
Let $\CX$ be as in Setup \ref{Setup} which is also closed under submodules and is contained in ${}^{\perp}A$. Then for every $X$ in $\CX$, the injective dimension of $\zeta^{\CX}_X$ is at most one.
\end{lemma}

\begin{proof}
Let $F \in \mmod\CX$ be an arbitrary object. Since $\CX$ is closed under submodules, we have a projective resolution
\[0 \lrt ( - ,X_2) \lrt ( - ,X_1) \lrt ( - ,X_0) \lrt F \lrt 0,\]
of $F$. For the proof, it is enough to show that the induced map
\[(( - ,X_1), \zeta^{\CX}_X) \lrt (( - ,X_2),\zeta^{\CX}_X)\]
is surjective, or in other words, $\Ext^2(F, \zeta^{\CX}_X)=0$. This follows easily using the same argument as in the proof of the above lemma.
\end{proof}

Let $T \in \mmod\La$ be a $\La$-module, where as usual $\La$ is an artin algebra. $T$ is called a tilting module if it satisfies the following conditions.
\begin{itemize}
\item [(a)] The projective dimension of $T$ is at most $1$;
\item [(b)] $T$ is rigid, i.e. $\Ext^1_{\La}(T,T)=0$;
\item [(c)] $T$ has $n$ indecomposable summands, where $n$ is the number of indecomposable direct summands of $\La$.
\end{itemize}

Dually $T \in \mmod\La$ is called a cotilting module if its injective dimension is at most $1$ and it satisfies the conditions (b) and (c) above.

\begin{remark}\label{AusEqu}
Let $\La$ be an artin algebra and $\CX$ be a contravariantly finite subcategory of $\mmod\La$ of finite type that contains projectives. Let $X$ be an additive generator of $\CX$, i.e. $\CX = \add(X)$, where $\add(X)$ denotes the set of all direct summands of finite direct sums of copies of $X$. Then $\End_{\La}(X)$ is called the $\CX$-Auslander algebra of $\La$ and is denoted by $\Aus(\CX)$, see Definition 6.3 of \cite{Be}. A well known fact is that $\Aus(\CX)$, up to Morita equivalence, is independent of the choice of the additive generator of $\CX$.
Moreover, there is an equivalence of abelian categories
\[e_X:\mmod\CX \rt \mmod\Aus(\CX)\]
defines for $F \in \mmod \CX$ by the evaluation of $F$ at $X$, denoted by $F(X)$. It is in a obvious way a right $\Aus(\CX)$-module.

In case $\CX=\mmod\La$, $\La$ is called of finite representation type and $\Aus(\CX)$ is called the Auslander algebra of $\La$, denoted by $\Aus(\La)$. In case $\CX=\Gprj\La$, $\La$ is called of finite Cohen-Macaulay type (CM-finite, for short). $\Aus(\Gprj\La)$ is then called the Cohen-Macaulay Auslander algebra of $\La$.
\end{remark}

\begin{theorem}\label{Main-1}
Let $\La$ be an artin algebra and $\CX$ be a contravariantly finite subcategory of $\mmod\La$ such that $\prj\La\subseteq \CX \subseteq {}^{\perp}\La$. Assume further that $\CX=\add(X)$ is of finite type. Set $\Ga=\Aus(\CX)$ and $T=\zeta^{\CX}_{X}(X)$. The following statements then hold true.
\begin{itemize}
\item[$(i)$] If $\CX$ is closed under syzygies, then $T$ is a tilting module over $\Ga$.
\item[$(ii)$] If, furthermore, $\CX$ is closed under submodules, then $T$ is a cotilting module over $\Ga$.
\end{itemize}
\end{theorem}

\begin{proof}
We just prove part $(i)$. Part $(ii)$ follows similarly. Conditions (a) and (b) of the definition of a tilting module follows from Lemmas \ref{Lem-pd} and \ref{Lem-rigid}, respectively. To see condition (c), note that similar to the proof of Lemma 2.2 of \cite{CS}, we can deduce that $\zeta^{\CX}$ is a full and faithful functor that preserves indecomposable objects. Hence (c) follows from the fact that for an indecomposable summand $M$ of $X$, $( - ,M)$ is an indecomposable projective object of $\mmod\CX$.
\end{proof}

Let $\La$ be an artin algebra such that both $\id_{\La}\La$ and $\id\La_{\La}$ are finite. It is known \cite{Z} that in this case, $\id_{\La}\La=\id\La_{\La}$, say $n$. Then $n$ is called the Gorenstein dimension of $\La$, denoted by $\Gdim(\La)$, and $\La$ is called the Iwanaga-Gorenstein algebra of Gorenstein dimension $n$, see \cite[Remark 2.4.6]{NRTZ}. Throughout the paper, we call such algebras Gorenstein of G-dimension $n$.

\begin{corollary}
Let $\La$ be a Gorenstein algebra of G-dimension $1$ which is of finite CM-type. Let $G$ be an additive generator of $\Gprj\La$. Then $\zeta^{\CG}_{G}(G)$ is both a tilting and a cotilting module over $\Aus(\Gprj \La)$, the Cohen-Macaulay Auslander algebra of $\La$.
\end{corollary}

\begin{proof}
In the above theorem set $\CX=\Gprj\La$. Just note that since $\La$ is of G-dimension $1$, $\Gprj\La$ is closed under submodules.
\end{proof}

Crawley-Boevey \cite{C-B} and later Schofield \cite{S} introduced the notion of quiver Grassmannians as varieties parametrizing subrepresentations of a quiver representation, in order to study the generic properties of representations of a quiver $\CQ$. Let $K$ be an algebraically closed field. By \cite{CFR}, if the quotient algebra $K\CQ/I$ of global dimension at most $2$,  admits a representation $M$ which is rigid and has both the injective and the projective dimension at most one, then the quiver Grassmannian ${\rm Gr}_e(M)$ is smooth and reduced, with irreducible and equidimensional connected components, where $e$ is a dimension vector for $\CQ$. See also \cite[Theorem 2.9]{CL}.
Using this result, Chen and Lu \cite[Theorem 3.10]{CL}, showed that if $K\CQ/I$ is a gentle algebra which is Gorenstein of G-dimension $1$, then its Cohen-Macaulay Auslander algebra is of global dimension at most two and admits a module $M$ with the above mentioned properties, and hence the quiver Grassmannian ${\rm Gr}_e(M)$ is smooth and reduced, where $e$ is a dimension vector for $\CQ^{\Aus}$. Our next corollary provides a generalization of this result. In particular, we remove the assumption that $\La$ is gentle.

\begin{corollary}\label{CL}
Let $\La$ be an artin algebra over an algebraically closed field. Let $\CX$ be as in Setup \ref{Setup} which is also closed under submodules, is contained in ${}^{\perp}\La$ and is of finite type. Let $\CQ$ denote the quiver of $\Aus(\CX)$ and $e$ be a dimension vector of it. Then the quiver Grassmannian ${\rm Gr}_e(\zeta^{\CX}_X(X))$ is smooth and reduced.
\end{corollary}

\begin{proof}
First note that since $\CX$ is closed under submodules, it is easily seen that the global dimension of $\Aus(\CX)$ is at most two. Now the result follows from \cite[Theorem 2.9]{CL}, in view of Lemmas \ref{Lem-pd}, \ref{Lem-rigid} and \ref{Lem-id}.
\end{proof}

We end this section by some examples. To this end we need to recall the following notions. Let $M$ and $N$ be $\La$-modules. We say that $N$ is generated, respectively cogenerated, by $M$ if there is an epimorphism $M^n \rt N$, respectively a monomorphism $N \rt M^n$, for some $n \in \N$. Let $\gen(M)$, respectively $\cogen(M)$, denote the full subcategory of $\mmod\La$ consisting of all modules that are generated, respectively cogenerated, by $M$.

\begin{example}\label{Example1}
\begin{itemize}
\item[(1)] Let $\La_n=k[x]/(x^n)$, where $k$ is an algebraically closed field. Then $\La_n$ is a self-injective algebra of finite representation type. In this case, $\zeta^{\La_n}_{M}(M)$ is the characteristic tilting module for the quasi-hereditary algebra $\Aus(\La)$, where $M$ is a basic additive generator for $\mmod \La_n$. See \cite[\S 2]{RZ} for more details.
\item[(2)] Let $\widehat{Q}$ be the direct sum of representatives of the isomorphism classes of all indecomposable projective-injective modules. Let $\mathcal{C}_{\La}:=\gen(\widehat{Q})\cap \cogen(\widehat{Q})$ be the subcategory of $\mmod \La$ consisting of all modules generated and cogenerated by $\mathcal{C}_{\La}.$  Crawley-Boevey and Sauter \cite{CS} proved that if $\rm{gldim}(\La)=2$, then the algebra $\La$ is an Auslander algebra if and only if there exists a tilting module $T$ in $\mathcal{C}_{\La}$. Furthermore, $T$ is the unique tilting module in $\mathcal{C}_{\La}$ and it is also a cotilting module. If $\La$ is a self-injective algebra of finite type, for basic addetive generator module $M$  let $\Gamma=\rm{End}_{\La}(T)$ denote the associated Auslander algebra. Then  $T \simeq \zeta^{\La}_{M}(M)$ is the unique tilting and cotilting module in $\mathcal{C}_{\Gamma}.$
\item[$(3)$] Let $\La$ be a basic hereditary artin algebra of finite type. Then in this case for basic additive generator $M$ of $\mmod \La$, $\zeta^{\La}_M(M)$ is just the trivial tilting module, i.e. $M=\La$ as a right $\La$-module.
\end{itemize}
\end{example}

\section{Relative Auslander algebras}
Assume, as before, that $\La$ is an artin algebra and $\CX$ is a contravariantly finite subcategory of $\mmod\La$ that contains projectives. Furthermore, assume that $\CX$ is of finite type with $X$ as an additive generator of it. In view of the equivalence $\mmod\CX \simeq \Aus(\CX)$ induced by the evaluation functor, Remark \ref{AusEqu}, we get the following recollement of $\mmod\Aus(\CX)$
\[\xymatrix{ \mmodd \Aus(\CX) \ar[rr]^{L}  && \mmod \Aus(\CX) \ar[rr]^{\mu=\va e^{-1}} \ar@/^1pc/[ll]^{L_{\rho}} \ar@/_1pc/[ll]_{L_{\la}} && \mmod A \ar@/^1pc/[ll]^{\mu_{\rho}} \ar@/_1pc/[ll]_{\mu_{\la}} }\]
where $\mmodd \Aus(\CX)$ denotes the kernel of $\va e^{-1}$.

The following commutative diagram
\[\xymatrix{\mmod\CX \ar[r]^{\va} \ar[d]^{e} & \mmod\La \ar[d]^1 \\ \mmod\Aus(\CX) \ar[r]^{ \ \ \va e^{-1}} & \mmod\La  }\]
in view of \ref{Eq-Recollement}, implies that this recollement is equivalent to $\CR(\CX,\La)$. Let us denote it by $\CR(\CX, \Ga)$, where $\Ga=\Aus(\CX)$.

By \ref{Rem-Bijection} a \TTF-triple is associated to this recollement. In our next theorem, we plan to study this \TTF-triple. Our strategy for this study is to transfer the problems related to $\mmod\Aus(\CX)$ to $\mmod\CX$, that has nicer homological properties, via the above mentioned equivalence, and then translate the results back to $\mmod\Aus(\CX)$.

Throughout, let $\CP^{\leq 1}(\Ga)$, resp. $\CI^{\leq 1}(\Ga)$, denote the full subcategory of $\mmod\Ga$ consisting of all objects of projective dimension, resp. injective dimension, at most $1$.

\begin{theorem}\label{Main12}
Let $\CX$ be as in Setup \ref{Setup} which is of finite type, is closed under submodules and is contained in ${}^{\perp}\La$. Set $T=\zeta^{\CX}_{X}(X)$, where $X$ is an additive generator of $\CX$. Consider the recollement $\CR(\CX, \Ga)$, where $\Ga=\Aus(\CX)$. Then the following statements hold.
\begin{itemize}
\item[$(i)$] $\cogen(T) \subseteq \Ker L_{\rho} = \CP^{\leq 1}(\Ga)$. Moreover, if $\La$ is self-injective then the inclusion becomes equality.
\item[$(ii)$] $\Ker L_{\la}= \gen(T) \subseteq \CI^{\leq 1}(\Ga)$. Moreover, if $\La$ is self-injective then the inclusion becomes equality.
\end{itemize}
\end{theorem}

\begin{proof}
By \cite[Corollary 4.4]{PV} we have a \TTF-triple $(\Ker L_{\La},L(\mmodd \Ga),\Ker L_{\rho})$ associated to the recollement $\CR(\CX,\Ga)$.

$(i)$ Since $\Ker L_{\rho}$ is a torsion-free class, it is closed under subobjects. Also obviously it is closed under finite direct sums. Hence to show that $\cogen(T) \subseteq \Ker L_{\rho}$ it is enough to show that $T \in \Ker L_{\rho}$. But this is equivalent to show that  $\zeta^{\CX}_X \in \Ker\ell_{\rho}$. We prove this later statement. To this end, take a projective presentation $P \rt Q \rt X \rt 0$ of $X$ in $\mmod \La.$ Then as in the proof of Lemma \ref{Lem-pd}, we obtain the following short exact sequence
$$0 \rt (-, \Omega(X)) \rt (-, Q) \rt \zeta^{\CX}_X \rt 0,$$ which is a projective resolution of $\zeta^{\CX}_X$ in $\mmod \CX.$ In view of this short exact sequence we get the following diagram
\[ \xymatrix@C-0.5pc@R-0.5pc{ &&& 0 \ar[d] &0 \ar@{.>}[d] & \\ 0 \ar[r] & 0 \ar[r] \ar@{=}[d] & (-, \Omega(X)) \ar@{=}[r] \ar@{=}[d] & (-, \Omega(X)) \ar[r] \ar[d] & 0 \ar@{.>}[d] \ar[r] & 0 \\ 0\ar[r] & 0 \ar[r] & (-,\Omega(X)) \ar[r] & (-, Q)\ar[d] \ar[r] & \zeta^{\CX}_X \ar[r] \ar@{.>}[dl] & 0 \\ &&& \va_{\rho}(X)=(-, X) \ar[d] & & \\ &&& \CL_X \ar[d] &&\\ &&& 0 && .}\]
In particular, we get a short exact sequence
$$\eta: 0 \rt \zeta^{\CX}_X \rt \va_{\rho}(X) \rt \CL_X \rt 0.$$
Now we apply the left exact functor $\ell_{\rho}$ on the short exact sequence $\eta$ and use the fact that $\ell_{\rho }\va_{\rho}=0$, to deduce that $\ell_{\rho}(\zeta^{\CX}_X)=0$. So we are done.

Now we prove that $\Ker L_{\rho} = \CP^{\leq 1}(\Ga)$. This is equivalent to prove that $\Ker\ell_{\rho}=\CP^{\leq 1}(\mmod\CX)$. To this end, let $F \in \Ker\ell_{\rho}$.
Part $(ii)$ of Lemma \ref{Lem-pd} implies that $( - ,\va(F))|_{\CX}$ has projective dimension at most $1$. So the exact sequence
\[\xymatrix{0 \ar[r] & F \ar[r] & ( - ,\va(F))|_{\CX} \ar[r] & F_1 \ar[r] & 0},\]
implies that $\pd F \leq 1$, because $\pd F_1 \leq 2.$ On the other hand, let $F \in \mmod\CX$ be such that $\pd F\leq 1$. The argument as in the previous part shows that $F \in \Ker\ell_{\rho}$. So we have the equality.

Finally, assume that $\La$ is a self-injective algebra and let $M \in \Ker L_{\rho}$. We show that $M \in \cogen(T)$. Equally, we can show that $F=e^{-1}(M) \in \Ker\ell_{\rho}$ belongs to $\cogen(\zeta^{\CX}_X)$. This we do. Since $\La$ is self-injective, $\va(F)$ can be embedded into a projective module $P\in\mmod\La$. So there exists a monomorphism
\[\xymatrix{( - ,\va(F))|_{\CX} \ar[r]  & ( - ,P),}\]
in $\mmod\CX$. Combining with the exact sequence
\[\xymatrix{0 \ar[r] & F \ar[r] & ( - ,\va(F))|_{\CX} \ar[r] & F_1 \ar[r] & 0,}\]
of Remark \ref{Interpretation}, we get an exact sequence $0 \lrt F \lrt ( - ,P)$ in $\mmod\CX$. This implies that $F \in \cogen(\zeta^{\CX}_X)$, because $( - ,P)=\zeta^{\CX}_P$.

$(ii)$ To prove this part, it is equivalent to prove that $\Ker\ell_{\la}=\gen(\zeta^{\CX}_X) \subseteq \CI^{\leq 1}(\mmod\CX).$ We prove this.
Let $F \in \mmod\CX$ and let $P \lrt Q \lrt \va(F) \lrt 0$  be a projective presentation of $\va(F)$. By definition, we get a projective presentation
\[\xymatrix{( - ,P) \ar[r] & ( - ,Q) \ar[r] & \va_{\la}\va(F) \ar[r] & 0,}\]
of $\va_{\la}\va(F)$ in $\mmod\CX$. Since $\zeta^{\CX}_Q=( - ,Q)$ and $\prj\La \subseteq \add(X)$, we get $\va_{\la}\va(F) \in \gen(\zeta^{\CX}_X)$.

Now let $F \in \Ker\ell_{\la}$. By \cite[Lemma/Definition 2.4]{CS} it is a costable object and hence the natural map $\va_{\la}\va(F) \lrt F$ is an epimorphism. Therefore $F \in \gen(\zeta^{\CX}_X)$.

For the converse, note that since $\Ker\ell_{\la}$ is a torsion class, it is closed under quotients. Hence to show that $\gen(\zeta^{\CX}_X) \subseteq \Ker\ell_{\la}$, it is enough to show that $\zeta^{\CX}_X \in \Ker\ell_{\la}$. Let $P \lrt Q \lrt X \lrt 0$ be a projective presentation of $X$. Consider the commutative diagram
\[\xymatrix{ & ( - ,P)  \ar[r] \ar[d] & ( - ,Q) \ar[r] \ar@{=}[d] & \va_{\la}\va(\zeta^{\CX}_X) \ar[r] \ar[d] & 0 \\ 0 \ar[r] & ( - ,\Omega(X)) \ar[r] & ( - ,Q) \ar[r] & \zeta^{\CX}_X \ar[r] & 0.}\]
Note that by \cite[Lemma 2.1(i)]{CS}  $\va\zeta^{\CX}_X \cong X$. Hence the natural map $\va_{\la}\va(\zeta^{\CX}_X) \lrt \zeta^{\CX}_X$ is an epimorphism. This in turn implies that $\zeta^{\CX}_X \in \Ker\ell_{\la}$.

To see the inclusion, let $F \in \gen(\zeta^{\CX}_X)$. By \cite[Corollary VI.2.6]{ASS} there exists exact sequence $\zeta_1 \st{\Psi}{\lrt} \zeta_0 \lrt F \lrt 0$, with $\zeta_1$ and $\zeta_0$ in $\add(\zeta^{\CX}_X)=\zeta^{\CX}(\add(X))$, as $\zeta^{\CX}$ commutes with direct sums \cite[Lemma 2.1(5)]{CS}.

Since $\zeta^{\CX}$ is full and faithful \cite[Lemma 2.2]{CS}, we deduce that there exists morphism $X_1 \st{\psi}{\lrt} X_0$ in $\mmod\La$ such that $X_0, X_1 \in \add(X)=\CX$, $\zeta^{\CX}_{X_i}=\zeta_i$, $i=0,1$, and $\zeta^{\CX}(\psi)=\Psi$. Since $\CX$ is closed under submodules, $\psi$ can be written as $\psi=i\pi$, where $X_1 \st{\pi}{\lrt} X'$ is an epimorphism and $X' \st{i}{\lrt} X_0$ is a monomorphism with $X' \in \CX$. But $\zeta^{\CX}$ preserves epimorphisms and monomorphisms \cite[Lemma 2.1(4)]{CS}, hence it follows that $\Psi$ also has an epi-mono decomposition $\zeta_1 \st{\zeta^{\CX}(\pi)}{\lrt} \zeta' \st{\zeta^{\CX}(i)}{\lrt} \zeta_0$ with $\zeta'=\zeta^{\CX}_{X'}$. So we get a short exact sequence $0 \lrt \zeta' \lrt \zeta_0 \lrt F \lrt 0$. By Lemma \ref{Lem-id}, the injective dimension of $\zeta'$ and $\zeta_0$ are at most $1$. Hence $\id F \leq 1$.

Finally assume that $\La$ is self-injective and $F \in \mmod\CX$ is of injective dimension at most $1$. By Remark \ref{Interpretation} we have an exact sequence associated to the recollement $\CR(\CX,\La)$,
\[\xymatrix{0 \ar[r] & F_0 \ar[r] & \va_{\la}\va(F) \ar[r] & F \ar[r] & F_1 \ar[r] & 0,}\]
such that $F_0, F_1 \in \mmodd\CX$. Let $K$ denote the $\Ker(F \lrt F_1)$. We have seen that $\va_{\la}\va(F) \in \gen(\zeta^{\CX}_X)$. Hence $K \in \gen(\zeta^{\CX}_X)$ and so $\id K \leq 1$. Therefore, since $\id F\leq 1$, we deduce that $\id F_1\leq 1$. Let
\[\xymatrix{ 0 \ar[r] & D(X_0, - ) \ar[r]^{D(f, -)} & D(X_1, - ) \ar[r] & 0}\]
be an injective resolution of $F_1$, with $X_0, X_1 \in \CX \subseteq {}^{\perp}\La=\mmod\La$. Since $F_1$ vanishes on projectives, for every $P \in \prj\La$, the map $D(X_0,P) \lrt D(X_1,P)$ is an isomorphism. Since $\La$  is self-injective, $\prj\La=\inj\La$ and so $X_1 \lrt X_0$ is an isomorphism. Therefore $F_1=0$ and hence $F \cong K \in \gen(\zeta^{\CX}_X)$. This completes the proof.
\end{proof}

Now we are ready to prove our main theorem in this section.

\begin{theorem}
Let $\La$ and $\La'$ be Gorenstein algebras of G-dimension $1$ which are of finite CM-type. Then $\La$ and $\La'$ are Morita equivalent if and only if their Cohen-Macaulay Auslander algebras, $\Ga=\Aus(\Gprj\La)$ and $\Ga'=\Aus(\Gprj\La')$, are Morita equivalent.
\end{theorem}

\begin{proof}
By \ref{Rem-Bijection},
\[(\mmodd\Gprj\La, \Ker\ell_{\rho}) \ \ {\rm and } \ \ (\mmodd\Gprj\La', \Ker\ell_{\rho}')\]
are torsion pairs. This, in view of the above theorem implies that
\[ \ \ \mmodd\Gprj\La={}^0\CP^{\leq 1}(\mmod \Gprj \La) \ \ {\rm and } \ \ \mmodd\Gprj\La'={}^0\CP^{\leq 1}(\mmod \Gprj \La'),\]
where for a class $\CY$ of abelian category  $\CA$, ${}^0\CY$ stands for the class of all $M \in \CA$ such that $\Hom_{\CA}(M,\CY)=0$. Now assume that $\Ga$ and $\Ga'$ are Morita equivalent, or equivalently $\mmod \Gprj \La$ and $\mmod \Gprj \La'$ are equivalent, say via $\Phi$. Since $\Phi$ is an exact functor, it maps $\CP^{\leq 1}(\mmod \Gprj \La)$ to $\CP^{\leq 1}(\mmod \Gprj \La')$ and hence maps ${}^0\CP^{\leq 1}(\mmod \Gprj \La)$ to ${}^0\CP^{\leq 1}(\mmod \Gprj \La')$. Therefore it induces an equivalence
\[\mmodd\Gprj\La \st{\Phi|}{\lrt} \mmodd\Gprj\La'.\]
This, in turn, implies that there exists an equivalence $\mmod\La \lrt \mmod\La'$.
Conversely, if $\La$ and $\La'$ are Morita equivalent, then $\Gprj \La \simeq \Gprj \La'$. Hence $\mmod \Gprj \La \simeq \mmod \Gprj \La'$, and so equivalently $\Ga$ and $\Ga'$ are Morita equivalent.
\end{proof}

\begin{remark}\label{KZ}
An artin algebra $\La$ is called CM-free if the inclusion $\prj\La \subseteq \Gprj\La$ is an equality. Kong and Zhang \cite{KZ} introduced a map $\Aus$ from the class of all CM-finite artin algebras up to Morita equivalence to the class of all CM-free artin algebras up to Morita equivalence. This map sends Gorenstein CM-finite algebras to algebras of finite global dimension. They showed that this map is surjective. Above theorem, in particular, implies that restricted to the Gorenstein algebras of G-dimension $1$, this map is injective.

However, it can not be injective on the class of Gorenstein algebras of G-dimension more than one. To see this, let $\La$ be a self-injective and non-semisimple artin algebra of finite type. Let $M$ be an additive generator of $\mmod \La$. Since $\La$ is self-injective, $\Gprj\La\simeq\mmod \La$. Moreover, $\Gprj\End_{\La}(M)\simeq \prj\End_{\La}(M)$, because $\End_{\La}(M)$ is an algebra of global dimension 2 or a Gorenstein algebra of G-dimension 2. Then $\Aus(\Gprj\La)$ and $\Aus(\Gprj\End_{\La}(M))$ are Morita equivalent but $\La$ and $\End_{\La}(M)$ are not Morita equivalent.
\end{remark}

\end{document}